\documentclass{amsart}

\usepackage{amsmath}
\usepackage{amssymb}
\usepackage[margin=3.5cm]{geometry}

\newtheorem{theorem}{Theorem}[section]

\newtheorem{definition}[theorem]{Definition}

\theoremstyle{definition}

\title{A note on the paper ``Contraction mappings in $b$-metric spaces'' by Czerwik}
\author{S\'andor Kaj\'ant\'o}
\author{Andor Luk\'acs}
\address{``Babe\c{s}-Bolyai'' University \\ Faculty of Mathematics and Computer Science \\
Kog\u{a}lniceanu Street No. 1 \\
400084, Cluj-Napoca \\
Romania}
\email{lukacs.andor@math.ubbcluj.ro}

\begin{document}

\begin{abstract}
In this paper we correct an inaccuracy that appears in the proof of Theorem 1. in Czerwik \cite{czerwik_1993}.
\end{abstract}

\maketitle

\section{Introduction}

It is well known that on certain function spaces the natural ``norm'' does not satisfy the triangle inequality (for example, on $L_p(\mathbb{R}^n)$ when $p\in(0,1)$ the functional that is usually called the ``$p$-norm'' has this property). In view of this observation, Czerwik in \cite{czerwik_1993} proposed a gene\-ra\-li\-sa\-tion of metric spaces by relaxing the triangle inequality in a way that allows the extension of fixed point theory to cover also these badly behaving function spaces. The resulting notion of $b$-metric spaces created a new direction in which fixed point theory could be developed and Czerwik was the first to generalise Banach's fixed point theorem for this case. Since then many authors contributed to this development and nowadays the field occupies a considerable position in fixed point theory.

The aim of our paper is to correct an inaccuracy that appears in the proof of Czerwik's theorem mentioned above. This occurs in the proof of Theorem 1 in \cite{czerwik_1993} (and also in the essentially same proof of Theorem 12.2 in \cite{kirk_fixed_2014}) at the step of proving that $(x_k)_{k\in\mathbb{N}}$ is a Cauchy sequence: the definition of $(x_k)_{k\in\mathbb{N}}$ depends on the choice of $\varepsilon$. However, in our proof we show that this inaccuracy can be corrected, and thus the conclusion of Czerwik's theorem remains true.

\section{Main result}
In this section we recall the notion of $b$-metric spaces, the statement of Theorem 1 in \cite{czerwik_1993} and we present our corrected proof.
\begin{definition}
	We say that \emph{$(X,d)$ is a $b$-metric space with constant $s\ge1$} if $d\colon X\times X\to [0,\infty)$ satisfies the following conditions for every $x,y,z\in X$:
\begin{itemize}
		\item[(i)] $d(x,y)=0$ if and only if $x=y$;
		\item[(ii)] $d(x,y)=d(y,x)$;
		\item[(iii)] $d(x,y)\le s[d(x,z)+d(z,y)]$.
\end{itemize}
\end{definition}
\begin{theorem}
	Let $(X,d)$ be a complete $b$-metric space with constant $s\ge1$ and suppose that $T\colon X\to X$ satisfies
	\begin{equation}
		d(Tx,Ty)\le\varphi(d(x,y))
		\label{eq:fi}
	\end{equation}
	for all $x,y\in X$, where $\varphi\colon[0,\infty)\to[0,\infty)$ is increasing and
	\[\lim_{n\to\infty}\varphi^n(t)=0\]
	for each $t\ge0$. Then $T$ has a unique fixed point $x^*\in X$ and $\lim_{n\to\infty} T^n(x)=x^*$ for each $x\in X$.
	\label{thm:main}
\end{theorem}
\begin{proof}
	Let $x\in X$ and define $x_{k}=T^kx$ for every $k\in\mathbb{N}$. Since for every
	$m,n\in\mathbb{N}$ we can apply $mn$ times \eqref{eq:fi} to get
	\[d(T^n x_{mn},x_{mn})\le\varphi^{mn}(d(x_n,x_0)),\]
	it follows that 
	\begin{equation}
	\lim_{m\to\infty}d(T^n x_{mn},x_{mn})=0, \quad\forall n\in\mathbb{N}.
		\label{eq:lim}
	\end{equation}
	In the next step for any $\varepsilon>0$ we construct an $\widetilde n$ and $\widetilde m$ such that 
	\begin{equation}
		T^{\widetilde n}(B(x_{\widetilde m \widetilde n},\varepsilon))\subseteq B(x_{\widetilde m \widetilde n},\varepsilon),
		\label{eq:sphere}
	\end{equation}
	where $B(x_{\widetilde m \widetilde n},\varepsilon)=\{u\in X\mid d(u,x_{\widetilde m \widetilde n})<\varepsilon\}$.
	First we choose $\widetilde n\in\mathbb{N}$ such that $\varphi^{\widetilde n}(\varepsilon)<\frac{\varepsilon}{2s}$. By \eqref{eq:lim}, for $\varepsilon$ and $\widetilde n$ we can choose an $\widetilde m\in\mathbb{N}$ such that 
	\[
		d(T^{\widetilde n}x_{m\widetilde n},x_{m\widetilde n})<\frac{\varepsilon}{2s},\quad \forall m\ge\widetilde m.
	\]
	The inclusion given in \eqref{eq:sphere} holds for these indices since for any $u\in B(x_{\widetilde m\widetilde n},\varepsilon)$  both inequalities 
	\[d(T^{\widetilde n}u,T^{\widetilde n}x_{\widetilde m\widetilde n})\le\varphi^{\widetilde n}(d(u,x_{\widetilde m\widetilde n}))\le\varphi^{\widetilde n}(\varepsilon)<\frac{\varepsilon}{2s}\]
	and
	\[d(T^{\widetilde n}x_{\widetilde m\widetilde n},x_{\widetilde m\widetilde n})\le \frac{\varepsilon}{2s}\]
	are satisfied, hence we can use the relaxed  triangle inequality to obtain
	\begin{align*}
	d(T^{\widetilde n} u, x_{\widetilde m\widetilde n})&\le s\big[d(T^{\widetilde n} u,T^{\widetilde n}x_{\widetilde m\widetilde n})+ d(T^{\widetilde n}x_{\widetilde m\widetilde n},x_{\widetilde m\widetilde n})\big]\\
	&<s\left[\frac{\varepsilon}{2s}+\frac{\varepsilon}{2s}\right]=\varepsilon.
	\end{align*}

	We observe that \eqref{eq:lim} also implies that there exists an $m_0\in\mathbb{N}$ such that 
	\begin{equation}
		d(x_{m\widetilde n},x_{m\widetilde n+p})<\varepsilon,\quad\forall m\ge m_0, \forall p\in\{0,1,\dots,\widetilde n-1\}.
		\label{eq:mp}
	\end{equation}
	Indeed, since \eqref{eq:lim} holds for $n=1$, there exists $k_0\in \mathbb{N}$ such that 
	\[d(x_{k+1},x_k)<\frac{\varepsilon}{\widetilde ns^{\widetilde n}},\]
	for all $k\ge k_0$.
	Let $m_0$ be such that $m_0\widetilde n>k_0$. We can apply $p-1$ times the relaxed triangle inequality to obtain
	\[
		d(x_{m\widetilde n},x_{m\widetilde n+p})\le\sum_{i=0}^p  s^{i+1} d(x_{m\widetilde n+i},x_{m\widetilde n+i+1})\le\sum_{i=0}^{\widetilde n-1}  s^{\widetilde n}\frac{\varepsilon}{\widetilde ns^{\widetilde n}}=\varepsilon,
	\]
	for all $m>m_0$ and $p\in\{0,...,\widetilde n-1\}$.
	
	We have prepared all the necessary technicalities  to prove that $(x_k)_{k\in\mathbb{N}}$ is a Cauchy sequence. 
	For any $\varepsilon>0$ first construct $\widetilde n,\widetilde m$ such that \eqref{eq:sphere} holds and then $m_0$ such that \eqref{eq:mp} is also satisfied. 
	Let $\overline m=\max\{\widetilde m, m_0\}$ and $k_1,k_2\in\mathbb{N}$ such that $k_1,k_2\ge \overline m \widetilde n$. 
	We can write $k_1=m_1 \widetilde n+p_1$, $k_2=m_2\widetilde n+p_2$, where $p_1,p_2\in\{0,...,n-1\}$ and $m_1,m_2\ge \overline m$.
	The construction of these indices implies
	\begin{align*}
		d(x_{k_1},x_{m_1\widetilde n})<\varepsilon \mbox{ and } d(x_{m_2\widetilde n},x_{k_2})<\varepsilon,&\quad\mbox{by \eqref{eq:mp}};\\
		d(x_{m_1\widetilde n},x_{\widetilde m\widetilde n})<\varepsilon\mbox{ and } d(x_{\widetilde m\widetilde n},x_{m_2\widetilde n})<\varepsilon,&\quad\mbox{by \eqref{eq:sphere}}.
	\end{align*} 
	Therefore we can use the relaxed triangle inequality to obtain
	\begin{align*}
		d(x_{k_1},x_{k_2})&\le s d(x_{k_1},x_{m_1\widetilde n})+s^2d(x_{m_1\widetilde n},x_{\widetilde m\widetilde n})\\
		&\qquad +s^3d(x_{\widetilde m\widetilde n},x_{m_2\widetilde n})+s^3d(x_{m_2\widetilde n},x_{k_2})\\
		&\le s\varepsilon+s^2\varepsilon+s^3\varepsilon+s^3\varepsilon\\
		&\le 4s^3\varepsilon.
	\end{align*}

	Thus we proved that $(x_k)_{k\in\mathbb{N}}$ is a Cauchy sequence.
	Since $(X, d)$ is complete, there exists an $x^*\in X$ such that $x_{k} \to x^*$. Then 
    \begin{align*}
    s^{-1} d(x^*, Tx^*)&\le \liminf_{k\to\infty} d(x_{k+1}, T x^*)\\
    &\le \limsup_{k\to \infty} d(x_{k+1}, T x^*) \\
    &= \limsup_{k\to\infty} d(T x_{k}, T x^*)\\
    &\le \limsup_{k\to \infty} d(x_{k}, x^*)= 0,
    \end{align*}
    hence $T x^* = x^*$.

    It remains to prove that $T$ has no other fixed point besides of $x^*$. Suppose that $y^*$ is also a fixed point of $T$. Hence
    \[0\le d(x^*,y^*)=d(Tx^*,Ty^*)\le\varphi(d(x^*,y^*)),\]
    and since $\varphi$ is increasing we have
    \[0\le d(x^*,y^*)\le\varphi^n(d(x^*,y^*)),\quad\forall n\in\mathbb{N}.\]
    If we let $n\to\infty$, the last inequality implies $d(x^*,y^*)=0$. 

\end{proof}

\bibliographystyle{plain}
\bibliography{myref}{}

\end{document}